\documentclass[12pt]{article}
\usepackage{latexsym,amsmath,amsthm,verbatim,ifthen,amssymb}
\usepackage{mathstuff}
\usepackage[all]{xy}

\usepackage[margin=1.5cm]{geometry}

\newtheorem{theorem}{Theorem}
\newtheorem{corollary}[theorem]{Corollary}
\newtheorem{lemma}[theorem]{Lemma}
\newtheorem{proposition}[theorem]{Proposition}
\newtheorem{example}[theorem]{Example}
\newtheorem{definition}[theorem]{Definition}
\newtheorem{remark}[theorem]{Remark}

\begin{document}
\title{Computations in rational sectional category\footnotetext{This work has been supported by FEDER through the Ministerio de Educaci\'on y Ciencia project MTM2010-18089.}}
\author{J.G. Carrasquel-Vera\footnote{Institut de Recherche en Math\'ematique et Physique, Universit\'e catholique de Louvain, 2 Chemin du Cyclotron, 1348 Louvain-la-Neuve, Belgium.
E-mail: \texttt{jose.carrasquel@uclouvain.be}}}

\date{}

\maketitle
\begin{abstract}
We give simple upper bounds for rational sectional category and use them to compute invariants of the type of Farber's topological complexity of rational spaces. In particular we show that the sectional category of formal morphisms reaches its cohomological lower bound and give a method to compute higher topological complexity of formal spaces in terms of their cohomology.
\end{abstract}

 \vspace{0.5cm}
 \vspace{0.2cm}

\maketitle

\section*{Introduction}
This paper concerns the rational sectional category of a continuous map $f\colon X\rightarrow Y$ and, in particular, the rational topological complexity of a space $X$. All the spaces considered will be supposed simply connected CW-complexes with finite Betti numbers.\\

Recall, \cite{Schwarz66}, that the sectional category of $f$, $\secat(f)$, is the least integer $m$ for which there is an open cover $\ca U_0,\ldots, U_m\cb$ of $Y$ and maps $s_i\colon U_i\rightarrow X$ such that $f\circ s_i$ is homotopic to the inclusion of $U_i$ in $Y$. When $X$ is contractible, the sectional category of $f$ is the usual LS category of $Y$, see \cite{biblels}.\\

We give special attention to the particular case, introduced by Y.B. Rudyak in \cite{Rudyak10}, of \emph{higher topological complexity} of a space $X$, $\tc_n(X)$, defined as the sectional category of the $n$-diagonal map $\Delta_n\colon X\rightarrow X^n$. The case $n=2$ yields M. Farber's well known \emph{topological complexity}, $\tc(X)$, introduced in \cite{Farber03}. Explicit computations for topological complexity of rational spaces can be found in \cite{Jessup12} and \cite{Lechuga07}, for instance.\\

Denote by $H ^*(X;R)$ the cohomology ring of $X$ with coefficients in a ring $R$. It is well known, \cite{biblels}, that \[\nil\ \ker(H ^*(f;R))\le\secat(f),\] where $\nil$ denotes the \emph{nilpotency} of an ideal, $\nil\ I:=\min\ca k\colon \ I^{k+1}=0\cb.$\\

We will denote $X_0$ the rationalisation of $X$ and $f_0\colon X_0\rightarrow Y_0$ the rationalisation of $f$. Following the scheme of Jessup-Murillo-Parent in \cite{Jessup12}, the following approximation to rational sectional category can easily be deduced:\\

Let $\varphi\colon A\rightarrow B$ be a surjective \cdga\ morphism. Define $\asecat(\varphi)$ as the smallest integer $m$ such that the quotient map \[\rho_m\colon (A,d)\rightarrow\pa\frac{A}{K^{m+1}}, \ol{d}\pb\] admits a homotopy retraction, where $K$ denotes the kernel of $\varphi$. If $f$ is a continuous map, define $\asecat(f)$ as the least $\asecat(\varphi)$ with $\varphi$ a surjective model for $f$.\\

Observe that, for rational LS category, the main theorem of \cite{FHT83} asserts that $\cat(X_0)=\asecat(*\hookrightarrow  X)$. Inspired on this, the obvious questions to ask whether $\secat(f_0)=\asecat(f)$. Although one of these inequalities is not known in general, the other one holds:
\begin{proposition}\label{prop11}
For any continuous map $f$, $\secat(f_0)\le \asecat(f)$.
\end{proposition}

This proposition is used to establish results regarding higher topological complexity including generalisations of \cite[Theorem 1.2]{Lechuga07} and \cite[Theorem 1.4]{Jessup12}. Namely, we prove that if $X$ is a formal space then \[\tc_n(X_0)=\nil\ \ker\pa\mu_n\colon H^*(X,\mQ)^{\otimes n}\rightarrow H^*(X,\mQ)\pb,\] being $\mu_n$ the multiplication. We also prove that if $X$ is a space such that $\pi_*(X)\otimes \mQ$ is finite dimensional and concentrated in odd degrees, then \[\tc_n(X_0)=(n-1)\cat(X_0).\]

Finally, we introduce the concept of homology nilpotency, $\hnil$, as an improvement of the upper bound $\nil\ \ker\ \varphi$. We also study the case that $H^*(f,\mQ)$ is surjective giving a dimensional upper bound for $\secat(f_0)$ and establishing \[\nil\ H(K)\le\secat(f_0)\le \hnil\ K.\]

\section{Rational sectional category}
Throughout this paper we will often use rational homotopy theory techniques for which much more than needed can be found in \cite{Bible}. We now present some basic facts. To every simply connected CW-complex of finite type $X$ one can associate a \emph{rationalisation} map $\rho\colon X\rightarrow X_0$ where $\pi_*(X_0)$ is a rational vector space and $\pi_*(\rho)\otimes\mQ$ is an isomorphism, the space $X_0$ is called the \emph{rationalisation} of $X$. This construction is functorial in the sense that a map $f\colon X\rightarrow Y$ can also be \emph{rationalised} to a map $f_0\colon X_0\rightarrow Y_0$ commuting with rationalisations.\\

On the other hand, every such $X$ has a (Sullivan) minimal model $\sul{V}$. This is a commutative differential graded algebra over $\mQ$ (\cdga\ for short), where $\Lambda V$ denotes the free graded commutative algebra on a graded vector space $V$ and where $d(V)\subset\Lambda^{\ge2}V,$ see \cite[Chapter 12]{Bible}. This correspondence yields an equivalence between the homotopy categories of rational 1-connected CW-complexes of finite type and 1-connected \cdga's of finite type. Moreover, every \cdga\  morphism $\varphi\colon (A,d)\rightarrow (B,d)$ admits a minimal relative model
\[\xymatrix@C=2cm{
(A,d)\ar@{>->}[dr]_{i}\ar[r]^{\varphi}&(B,d)\\
&(A\otimes\Lambda V,d)\ar[u]^{\simeq}_{\psi}\\
}\]
where $i$ is the canonical injection, $\psi\circ i=\varphi$, \[d(V)\subset \pa A^+\otimes\Lambda V\pb+ \pa A\otimes\Lambda^{\ge 2}V\pb,\] and $\psi$ is a quasi-isomorphism.\\

Let $i\colon X\cofib Y$ be a cofibration, the $m$-fat-wedge of $i$ is the subspace of $Y^{m+1}$ defined as \[T^m(i):=\ca (x_0,\ldots,x_m)\in Y^{m+1}\colon x_j\in i(X)\mbox{ for some } j=0,\ldots, m\cb,\] with inclusion $W_m(i)\colon T^m(i)\rightarrow Y^{m+1}$.

\begin{theorem}(\cite{Fasso})\label{th:whiteheadCharSecat}
Let $f$ be a map and $i\colon X\rightarrow Y$ a cofibration replacement for $f$ with $Y$ a paracompact space. Then $\secat(f)$ is the smallest $m$ such that there exists a map $r$ making the following diagram homotopy commutative:
\[\xymatrix@C=2cm{
&T^{m}(i)\ar[d]^{W_m(i)}\\
Y\ar@{-->}[ur]^r\ar[r]_-{\Delta_{m+1}}&Y^{m+1}.\\
}\]
\end{theorem}

\begin{remark}
In order to take care of unnecessary technical conditions, we will consider the statement of previous theorem as the definition for sectional category.
\end{remark}

Now let $L$ be a simplicial complex on $m$ vertices, then the \emph{polyhedral product} of the pair $(Y,X)$ and $L$ is defined as \[\underline{(Y,X)}^L=\bigcup_{\sigma\in L}\pa \prod_{j=1}^mA_\sigma^j\pb\subset X^m,\] where \[A_\sigma^j:=
\begin{cases}
Y & \mbox{ if } j\in\sigma\\
X & \mbox{ if } j\not\in \sigma.\\
\end{cases}\] The $m$-fat-wedge $T^{m}(i)$ can be written in terms of the polyhedral product $\underline{(Y,X)}^{\partial \Delta^m}$. Therefore, if $A\rightarrow B$ is a surjective \cdga\  model for $i$ with kernel $K$, then \cite[Thm. 1]{Felix09} tells us that a model for the inclusion $W_m(i)\colon T^{m}(i)\cofib Y^{m+1}$ is the projection \[q_m\colon A^{\otimes m+1}\longrightarrow \frac{A^{\otimes m+1}}{K^{\otimes m+1}}.\] Recall also that if $A$ is a \cdga\ model for $Y$ then the diagonal map $\Delta_{m+1}\colon Y\rightarrow Y^{m+1}$ is modelled by the multiplication morphism $\mu_{m+1}\colon A^{\otimes m+1}\rightarrow A$. These remarks lead us to
\begin{definition}\index{sectional category in \cdga}
The \emph{sectional category} of a surjective \cdga\ morphism $\varphi\colon A\rightarrow B$, $\secat(\varphi)$, is the smallest $m$ for which there exists a \cdga\ morphism $\tau$ such that $\tau\circ i_m= \mu_{m+1},$
\[\xymatrix{
A^{\otimes m+1}\ar[dd]_{\mu_{m+1}}\ar[r]^{q_m} \ar@{>->}[dr]^{i_m}&\frac{A^{\otimes m+1}}{K^{\otimes m+1}}\\
 &(A^{\otimes m+1}\otimes\Lambda W_{m}, D)\ar[u]_{\simeq}\ar@{-->}[dl]^{\tau}\\
A,\\
}\]
where $K=\ker\ \varphi$ and $i_m$ is a relative Sullivan model for $q_m$. The \emph{sectional category} of any morphism is defined as the sectional category of any of its surjective replacements.
\end{definition}

Taking the pushout
\[\xymatrix{
A^{\otimes m+1}\ar[d]_{\mu_{m+1}}\ar@{>->}[r]^-{i_m}&\pa A^{\otimes m+1}\otimes\Lambda W_{m},D\pb\ar[d]\\
A\ar@{>->}[r]_-{j_{m}}&\pa A\otimes\Lambda W_{m},\ol{D}\pb
}\]
one can easily check, thanks to pushout's universal property, that $\secat(\varphi)\le m$ if and only if $j_m$ admits a retraction. In fact, if $\varphi$ models a map $f$ then $j_m$ is a model for the $m$-th Ganea map, $G_m(f)$.

\begin{definition}\label{def:mcsecathsecat}\index{module sectional category}\index{homology sectional category}
Let $\varphi$ be a surjective \cdga\ morphism and consider previous diagram. 
\begin{itemize}
\item[(i)] The \emph{module sectional category} of $\varphi$, $\msecat(\varphi)$, is the smallest $m$ such that $j_m$ admits an $A$-module retraction,
\item[(ii)] the \emph{homology sectional category} of $\varphi$, $\hsecat(\varphi)$, as the smallest $m$ such that $H(j_m)$ is injective.
\item[(iii)] If a continuous map $f$ is modelled by $\varphi$, define $\msecat(f):=\msecat(\varphi)$ and $\hsecat(f):=\hsecat(\varphi)$.  
\end{itemize}
\end{definition}

The module sectional category of this paper coincides with the one introduced in \cite{FGKV06}. The expected cohomological lower bound follows:

\begin{proposition}
Let $\varphi\colon A\rightarrow B$ be a surjective \cdga\ morphism. Then \[\nil\ \ker\ H (\varphi)\le\hsecat(\varphi).\]
\end{proposition}

\begin{proof}
Suppose $\hsecat(\varphi)=m$, then $H(j_m)$ is injective, where $j_m$ is as in previous diagram. Let $[x_0],\ldots,[x_m]\in \ker\ H (\varphi)$, since $\varphi$ is surjective, there are $a_0,\ldots, a_m\in A$ such that $x_i-da_i\in K$, for $i=0,\ldots, m$. We have constructed a cycle $z:=(x_0-da_0)\otimes\cdots\otimes(x_m-da_m)\in K^{\otimes m+1}$ therefore there exists $\xi\in A^{\otimes m+1}\otimes\Lambda W_{m}$ such that $D\xi=z$. Thus $H(j_m)([x_0]\cdots[x_m])=[0]$ and since $H(j_m)$ is injective, $[x_0]\cdots[x_m]=[0]$, proving that $\nil\ \ker\ H (\varphi)\le m$.
\end{proof}

The following chain of inequalities is now clear for a surjective morphism $\varphi$, \[\nil\ \ker\ H(\varphi)\le \hsecat(\varphi)\le\msecat(\varphi)\le\secat(\varphi).\]

We now prove:

\begin{theorem}\label{th:rationaModelSecat}
Let $\varphi$ be a surjective \cdga\  model for $f$, then \[\secat(f_0)=\secat(\varphi).\]
\end{theorem}

\begin{proof}
Transform $f_0$ into a cofibration $i\colon X\rightarrow Y$. Denote by $\alpha \colon (\Lambda V,d)\quism A$ be a minimal model for $A$. By construction $(\Lambda V,d)$ is a minimal model of $Y$, and as explained at the beginning of this section, if $K=\ker\ \varphi$, then the projection \[q_m\colon A^{\otimes m+1}\rightarrow \frac{A^{\otimes m+1}}{K^{\otimes m+1}}\] is a model for the map $W_m(i)\colon T^m(i)\rightarrow Y^{m+1}$. Denote by $((\Lambda V)^{\otimes m+1}\otimes \Lambda W_{m},D)$ a relative model for $q_m$. Then we have a commutative diagram
\[\xymatrix@C=2cm{
A^{\otimes m+1}\ar[r]^-{q_m}&\frac{A^{\otimes m+1}}{K^{\otimes m+1}}\\
(\Lambda V)^{\otimes m+1}\ar[u]^\simeq\ar[d]_\simeq\ar@{>->}[r]&((\Lambda V)^{\otimes m+1}\otimes \Lambda W_{m},D)\ar[u]_\simeq\ar[d]^\simeq\\
\apl(Y^{m+1})\ar[r]_-{\apl(W_m(i))}&\apl(T^m(i))\\
}\]

Suppose $\secat(f_0)=m$, then, by Theorem \ref{th:whiteheadCharSecat}, there is $\theta\colon Y\rightarrow T^m(i)$ such that the following diagram homotopy commutes
\[\xymatrix{
T^m(i)\ar[rr]^-{W_m(i)}&&Y^{m+1}\\
&Y.\ar[ul]^\theta\ar[ur]_\Delta&\\
}\]
Then the relative lifting lemma gives a morphism $\gamma$ making commutative the upper triangle and homotopy commutative the lower triangle in

\[\xymatrix@C=2cm{
(\Lambda V)^{\otimes m+1}\ar[rr]^{\mu_{m+1}}\ar@{>->}[d]&& \Lambda V\ar[d]_{\simeq}\\
(\Lambda V)^{\otimes m+1}\otimes\Lambda W_{m}\ar@{-->}[urr]^{\alpha}\ar[r]_-{\simeq}&\apl(T^{m}(i))\ar[r]_-{\apl(\theta)}& \apl(Y).\\
}\]
The desired $\tau$ is given by pushout's universal property:
\[\xymatrix@C=2cm{
(\Lambda V)^{\otimes m+1}\ar@{}[dr]|{po}\ar@{>->}[r]\ar[d]_{\simeq}&(\Lambda V)^{\otimes m+1}\otimes\Lambda W_{m}\ar[d]^{\simeq}\ar@/^1pc/[ddr]^{\alpha\circ \gamma}\\
A^{\otimes m+1}\ar@/_1pc/[drr]_{\mu_{m+1}}\ar@{>->}[r]&A^{\otimes m+1}\otimes\Lambda W_{m}\ar@{-->}[dr]^{\tau}\\
&&A.\\
}\]
This proves that $\secat(\varphi)\le \secat(f_0)$. For the second inequality, just apply spatial realization functor, \cite[Chapt. 17]{Bible}. 
\end{proof}

Previous theorem combined with \cite[Thm. 23]{Diaz12} gives
\begin{corollary}
Given $f$ a continuous map, then \[\secat(f_0)\le\secat(f).\]
\end{corollary}

Then we have for a map $f$ that
\[\nil\ \ker H^*(f,\mQ)\le \hsecat(f)\le\msecat(f)\le\secat(f_0)\le\secat(f).\]

\section{The invariant \asecat(f)}\label{sec:asecat}
Recall that a \cdga\ morphism $\psi\colon A\rightarrow B$ admits a homotopy retraction if there exists a map $r\colon (A\otimes\Lambda V, D)\rightarrow A$ such that $r\circ i=\id_A$, where $i\colon A\cofib (A\otimes\Lambda V,D)$ is a relative Sullivan model for $\psi$. We now introduce the following upper bound to rational sectional category:

\begin{definition}\label{def:asecat}
Let $\varphi\colon A\rightarrow B$ be a surjective \cdga\ morphism with kernel $K$ and consider the projection \[\rho_m\colon (A,d)\rightarrow\pa\frac{A}{K^{m+1}}, \ol{d}\pb.\]Define:
\begin{itemize}
\item[(i)] $\asecat(\varphi)$ as the smallest integer $m$ such that $\rho_m$ admits a homotopy retraction,
\item[(ii)] $\amsecat(\varphi)$ the smallest $m$ such that $\rho_m$ admits a homotopy retraction as $A$-modules,
\item[(iii)] $\ahsecat(\varphi)$ the smallest $m$ such that $H(\rho_m)$ is injective.
\end{itemize}
\end{definition}

If $X$ is a space modelled by $\sul{V}$ and $\epsilon\colon\sul{V}\rightarrow \mQ $ is the augmentation then $\amsecat(\epsilon)$ is the classical module category of $X_0$ and $\ahsecat(\epsilon)$ is the rational Toomer invariant of $X$. Also, if $\mu\colon \sul{V}\otimes\sul{V}\rightarrow \sul{V}$ is the multiplication, then $\asecat(\mu)={\rm \bf tc}(X)$ and $\amsecat(\mu)={\rm \bf mtc}(X)$, as defined in \cite{Jessup12}.\\

Observe that $\asecat(\varphi)$, $\amsecat(\varphi)$ and $\ahsecat(\varphi)$ are not invariants of the weak homotopy type of $\varphi$. This can be seen explicitly in
\begin{example}\label{ex4.9}
Consider $A:=(\Lambda(a,b)/(a^2),d)$ the \cdga\ defined as $|a|=4$, $|b|=3$, $db=a$ and $\varphi\colon A\rightarrow \mQ$ the augmentation. Since $B:=(\Lambda v_7,0)$ is a minimal model for $A$, the augmentation $\psi\colon B \rightarrow \mQ$ is weakly equivalent to $\varphi$. It is easy to see that $\secat(\psi)=1$ while $\ahsecat(\varphi)\ge 2$.
\end{example}
This definition extends to continuous maps:

\begin{definition}
Let $f$ be a continuous map. Define:
\begin{itemize}
\item[(i)]$\asecat(f)$ as the least $\asecat(\varphi)$ with $\varphi$ a surjective model for $f$,
\item[(ii)] $\amsecat(f)$ as the smallest $\amsecat(\varphi)$ with $\varphi$ a surjective model for $f$,
\item[(iii)] $\ahsecat(f)$ as the smallest $\ahsecat(\varphi)$ with $\varphi$ a surjective model for $f$.
\end{itemize}
\end{definition}

\begin{proposition}\label{prop1}
For any surjective \cdga\ morphism $\varphi$, we have
\begin{itemize}
\item[(i)]$\secat(\varphi)\le \asecat(\varphi)$, 
\item[(ii)]$\msecat(\varphi)\le\amsecat(\varphi)$,
\item[(iii)]$\hsecat(\varphi)\le\ahsecat(\varphi)$.
\end{itemize}
\end{proposition}

\begin{proof}
Denote by $(A\otimes\Lambda Z_{m},D)\quism \frac{A}{K^{m+1}}$ a relative Sullivan model for $\rho_m$. Since multiplication induces a map \[\ol{\mu}\colon \frac{A^{\otimes m+1}}{K^{\otimes m+1}}\longrightarrow \frac{A}{K^{m+1}},\] the relative lifting lemma gives a morphism $\alpha$ making commutative the diagram 
\[\xymatrix{
A^{\otimes m+1}\ar[r]^{\mu_{m+1}}\ar@{>->}[d]&A\ar@{>->}[r]& A\otimes\Lambda Z_{m}\ar@{->>}[d]^{\simeq}\\
A^{\otimes m+1}\otimes\Lambda W_{m}\ar@{-->}[urr]^{\alpha}\ar[r]_-{\simeq}&\frac{A^{\otimes m+1}}{K^{\otimes m+1}}\ar[r]_(0.45){\ol{\mu}}& A/K^{m+1},\\
}\]
If $r\colon \pa A\otimes\Lambda Z_{m},D\pb\rightarrow A$ is a homotopy retraction for $\rho_m$ then the desired map $\tau$ is given by $r\circ\alpha$.
\end{proof}

The following corollary includes Proposition \ref{prop11}.
\begin{corollary}
If $f$ is a continuous map, then
\begin{itemize}
\item[(i)]$\secat(f_0)\le \asecat(f)$,
\item[(ii)]$\msecat(f)\le\amsecat(f)$,
\item[(iii)]$\hsecat(f)\le\ahsecat(f)$.
\end{itemize}

\end{corollary}

We now prove that for computing $\asecat(f)$ we can restrict to models for $f$ between Sullivan algebras and the answer does not depend on the choice of the model. The following lemma is straightforward.

\begin{lemma}\label{lem:psiRetractSoDoesPhi}
Consider the commutative \cdga\ diagram where $\omega$ is a quasi-isomorphism,
\[\xymatrix{
A\ar[d]_\varphi\ar[r]^-\omega_-\simeq&B\ar[d]\ar[d]^\psi\\
C\ar[r]&D.\\
}\]
If $\psi$ admits a homotopy retraction, then so does $\varphi$.
\end{lemma}

We can now prove
\begin{lemma}
Let $\varphi\colon A\rightarrow B$ be a surjective \cdga\  morphism and $\psi\colon \sul{T}\quism A$ a surjective Sullivan model for $A$. Then $\asecat(\varphi\circ\psi)\le\asecat(\varphi)$, $\amsecat(\varphi\circ\psi)\le\amsecat(\varphi)$ and $\ahsecat(\varphi\circ\psi)\le\ahsecat(\varphi)$.
\end{lemma}

\begin{proof}
The morphism $\psi$ induces a diagram
\[\xymatrix{
\Lambda T\ar[r]_-{\simeq}^-{\psi}\ar[d]&A\ar[d]\\
\frac{\Lambda T}{L^m}\ar[r]&\frac{A}{K^m},\\
}\] 
where $L$ denotes the kernel of $\varphi\circ\psi$. The result follows by previous lemma.
\end{proof}

\begin{lemma}
Let $\varphi\colon \sul{V}\rightarrow B$ be a surjective \cdga\  morphism and consider $\phi$ an extension of $\varphi$,
\[\xymatrix{
\sul{V}\ar@{>->}[dr]_{\simeq}\ar@{->>}[rr]^{\varphi}&&B\\
&\sul{V}\otimes\sul{W}.\ar@{->>}[ur]_{\phi}&
}\]
Then $\asecat(\varphi)=\asecat(\phi)$, $\amsecat(\varphi)=\amsecat(\phi)$ and $\ahsecat(\varphi)=\ahsecat(\phi)$.
\end{lemma}

\begin{proof}
Remark that $W$ admits a basis of the form $\ca v_i,w_i\cb$ with $dv_i=w_i$ and that, by a change of variable, one can suppose that $\phi(W)=0$. Now define on $\Lambda W$ a derivation $s$ of degree $-1$ by $s(w_i)=v_i$ and $s(v_i)=0$. For each $l\ge 1$, \[s\circ d+d\circ s\colon \Lambda^l W\rightarrow \Lambda^l W\] is multiplication by $l$ and thus $H(\Lambda^lW,d)=0$. Now, denoting $K=\ker \varphi$ and $L=\ker \phi$, we have that $L=K\oplus\Lambda V\otimes\Lambda^+W$ and $L^m=K^m\oplus I$ with \[I=\sum_{l=0}^{m-1}K^l\otimes \Lambda^{\ge m-l}W,\] with $K^0:=\Lambda V$. Since, as a vector space, \[I=K^{m-1}\otimes\Lambda^+W\oplus \frac{K^{m-2}}{K^{m-1}}\otimes\Lambda^{\ge 2}W\oplus\cdots\oplus\frac{K}{K^2}\otimes\Lambda^{\ge m-1}W\oplus \frac{\Lambda V}{K}\otimes\Lambda^{\ge m}W,\] an inductive argument shows that $H(I)=0$. As the five lemma gives a diagram

\[\xymatrix{
\Lambda V\ar[d]\ar@{>->}[rr]^-{\simeq}&&\Lambda V\otimes\Lambda W\ar[d]\\
\frac{\Lambda V}{K^m}\ar[rr]^-{\simeq}&&\frac{\Lambda V\otimes\Lambda W}{L^m},\\
}\] the lemma follows.
\end{proof}

\begin{corollary}
Let $f$ be a continuous map and $\varphi$ be a surjective model for $f$ between Sullivan algebras. Then $\asecat(f)=\asecat(\varphi)$, $\amsecat(f)=\amsecat(\varphi)$, $\ahsecat(f)=\ahsecat(\varphi)$.
\end{corollary}

\section{The case $H(\varphi)$ surjective}
Suppose $\varphi\colon A\rightarrow B$ is a surjective morphism with $H (\varphi)$ also surjective and write $K=\ker\ \varphi$. Then the short exact sequence
\[\xymatrix{
0 \ar@{^{(}->}[r]&K \ar@{^{(}->}[r]&A \ar@{->>}[r]&B \ar@{->>}[r]&0\\
}\]
yields the short exact sequence
\[\xymatrix{
0 \ar@{^{(}->}[r] &H (K) \ar@{^{(}->}[r]&H (A) \ar@{->>}[r]&H (B) \ar@{->>}[r]&0,\\
}\]
which tells us that $\nil\ \ker H (\varphi)=\nil\ H (\ker \varphi) $. Moreover, the homology of the projection \[q_m\colon A^{\otimes m+1}\longrightarrow \frac{A^{\otimes m+1}}{K^{\otimes m+1}}\] is given by \[H(q_m)\colon H(A)^{\otimes m+1}\longrightarrow \frac{H (A)^{\otimes m+1}}{H (K)^{\otimes m+1}}.\]

\begin{example}\label{exp1}
Consider the surjective morphism \[\varphi:\pa\Lambda(a_3,b_3,x_5);\ dx=ab\pb\longrightarrow (\Lambda(a,b)/(ab),0)\] whose kernel is $K:=(ab,x)$. We have that $\asecat(\varphi)=2$ since the projection $\rho_1\colon \Lambda(a,b,x)\rightarrow\frac{\Lambda(a,b,x)}{(abx)}$ is not injective in homology and $\nil\ K=2$. On the other hand, because of previous remarks and the fact that $H(K)=H^{\ge 8}(K)$, we have a commutative diagram
\[\xymatrix{
(\Lambda(a,b,x))^{\otimes 2}\ar[dd]_{\mu_2}\ar[r] \ar@{>->}[dr]&\frac{(\Lambda(a,b,x))^{\otimes 2}}{K^{\otimes 2}}\\
 &(\Lambda(a,b,x;d)^{\otimes 2}\ar[u]_{\simeq}\otimes\Lambda W_1, D)\ar[dl]^{\tau}\\
\Lambda(a,b,x;d)\\
}\]
with $W_1=W_1^{\ge 15}$ and $\tau(W_1)=0$. This shows that $\secat(\varphi)=1<2=\asecat(\varphi)$.
\end{example}

The idea for computing $\secat(\varphi)$ in the previous example can be generalised:

\begin{proposition}
Let $\varphi\colon A\rightarrow B$ be a surjective \cdga\ morphism such that $H(\varphi)$ is also surjective, $A=A^{< l}$ and $H(K)=H^{\ge k}(K)$. Then \[\secat(\varphi)\le \frac{l+1}{k}.\]
\end{proposition}

\begin{proof}
In this case, since $W_{m}=W_{m}^{\ge (m+1)k-1}$, a morphism $r$ making the diagram

\[\xymatrix{
A^{\otimes m+1}\ar[dd]_{\mu_{m+1}}\ar[r] \ar@{>->}[dr]^{i}&\frac{A^{\otimes m+1}}{K^{\otimes m+1}}\\
 &(A^{\otimes m+1}\otimes\Lambda W_{m}, D)\ar[u]_{\simeq}\ar[dl]^{\tau}\\
A,\\
}\] 
commute can be defined as $r(a):=\mu_{m+1}(a)$, for $a\in A^{\otimes m+1}$ and $r(W_{m}):=0$.
\end{proof}

\section{Homology nilpotency of an ideal}
Consider a surjective \cdga\  morphism $\varphi\colon A\fib B$ with $K=\ker \varphi$, then, by Proposition \ref{prop1}, $\secat(\varphi)\le \nil\ K$ but when $A$ is a Sullivan algebra then it is very \emph{likely} that $\nil\ K=\infty$. 

\begin{definition}\index{homology nilpotency}
Let $I$ be an ideal of a \cdga\  $A$, the \emph{homology nilpotency} of $I$ is \[\hnil\ I:=\min\ca k\colon I^{k+1}\subset J,\ J \mbox{ acyclic ideal of } A\cb.\]
\end{definition}

Remark that if $K^{m+1}$ is included an acyclic ideal $J$ of $A$ then we have a commutative diagram
\[\xymatrix{
A\ar@{->>}[r]^(0.41){\rho_m}\ar[dr]_{\simeq}&\frac{A}{K^{m+1}}\ar[d]\\
&\frac{A}{J}\\
}\]
which can be used to deduce a homotopy retraction of $\rho_m$. As a consequence we have
\begin{proposition}
Let $\varphi$ be a surjective \cdga\ morphism with $K:=\ker \varphi$. Then, 
\[\nil \ker H(\varphi)\le\secat(\varphi)\le\asecat(\varphi)\le \hnil\ K\le \nil\ K,\] and, if $H(\varphi)$ is surjective, 
\[\nil\ H(K)\le \secat(\varphi)\le\hnil(K).\]
\end{proposition}

\begin{example}
Consider $A=(\Lambda a_2,x; dx=a^2)$ and $\epsilon$ the augmentation on $A$. Then, since $K^2\subset (a^2,x)$, $\cat(A)=\secat(\epsilon)=\asecat(\epsilon)=\hnil(K)=1$, and $\nil(K)=\infty$.
\end{example}



\section{Sectional category of formal morphisms}
As the LS category of formal spaces, the sectional category of formal maps is very easy to compute.
\begin{definition}
A \cdga\ morphism $\varphi$ is said to be \emph{formal} if it is weakly equivalent to $H(\varphi)$. A continuous map is said to be \emph{formal} if it admits a formal \cdga\  model.
\end{definition}

For more on formal morphisms the reader is referred to \cite{FT88} and \cite{Vig79}. It is obvious from this definition that if $\varphi\colon  A\rightarrow B$ is formal then both $A$ and $B$ are formal as well.

\begin{theorem}\label{mainformal}
Let $\varphi\colon  A\rightarrow B$ be a formal morphism with $H (\varphi)$ surjective. Then \[\secat(\varphi)=\nil(\ker H (\varphi)).\]
\end{theorem}

\begin{proof}
By formality, $\secat(\varphi)=\secat(H (\varphi))$. Write $m:=\nil(\ker H (\varphi))$, we must prove that $\secat(H (\varphi))\le m$ but this is direct consequence of Proposition \ref{prop1} and the fact that $\pa \ker H (\varphi)\pb^{m+1}=\ca 0\cb$.
\end{proof}

\section{Applications}
As a direct consequence of Theorem \ref{th:rationaModelSecat} we have a rational model for topological complexity:
\begin{proposition}
Let $A$ be a \cdga\ model for a space $X$ and $\mu_n\colon A^{\otimes n}\rightarrow A$ the $n$-multiplication. Then \[\tc_n(X_0)=\secat(\mu_n).\]
\end{proposition}
\begin{proof}
Since rationalisation commutes with limits, we have \[\tc_n(X_0)=\secat(\Delta_{X_0}^n)=\secat((\Delta_X^n)_0)=\secat(\mu_n).\]
\end{proof}

We extend this proposition to
\begin{definition}
Let $X$ be a space, then define
\begin{itemize}
\item[(i)] $\mtc_n(X):=\msecat(\Delta_n)$,
\item[(ii)] $\htc_n(X):=\hsecat(\Delta_n)$.
\end{itemize}
\end{definition}

It was proven by L. Lechuga and A. Murillo in \cite{Lechuga07} that the topological complexity of formal spaces equals the nilpotency of the kernel of the multiplication $H^*(X,\mQ)\otimes H^*(X,\mQ)\rightarrow H^*(X,\mQ)$. We extend this result:
\begin{theorem}\label{th:tcnFormal}
If $X$ is a formal space, then \[\tc_n(X_0)=\nil\ \ker\pa \mu_n\colon H ^*(X,\mQ)^{\otimes n}\rightarrow H ^*(X,\mQ)\pb,\] and thus $\tc_n(X_0)=\mtc_n(X_0)=\htc_n(X_0)$.
\end{theorem}
\begin{proof}
If $X$ is a formal space, $\Delta_{X}^n$ is a formal map modelled by $\mu_n$. The result follows from Theorem \ref{mainformal}.
\end{proof}

Observe now that if $G$ is a subset of a \cdga\ $A$ and $I$ the ideal of $A$ generated by $G$ then $\nil\ I$ is the largest $m$ for which there exist $x_1,\ldots,x_m\in G$ such that $x_1\cdots x_m\ne 0$. Also, if $A^+$ is generated by $\ca x_i\cb_{i\in I}$, then $K_n$, the kernel of the $n$-multiplication morphism $\mu_n\colon A^{\otimes n}\rightarrow A$, is generated by $\ca x_{i,j}-x_{i,1}:\ i\in I,\ 2\le j\le n\cb$ where \[x_{i,j}=1\otimes\cdots\otimes 1\otimes x_i\otimes 1\otimes\cdots\otimes 1\in A^{\otimes j-1}\otimes A^+\otimes A^{\otimes n-j}.\] This is consequence of the fact that an element $x_1\otimes\cdots\otimes x_n$ can be written in the form\index{kernel of multiplication} 
\begin{align*}
\pa x_{1,1}\cdots x_{n-1,n-1}\pb \cdot& \pa x_{n,n}-x_{n,n-1}\pb +\\
\pa x_{1,1}\cdots x_{n-2,n-2}\pb \cdot& \pa x_{n-1,n-1}x_{n,n-1}-x_{n-1,n-2}x_{n,n-2}\pb +\\
\vdots&\\
\pa x_{1,1}\pb \cdot&\pa x_{2,2}\cdots x_{n,2}-x_{2,1}\cdots x_{n,1}\pb +\\
\ &(x_{1,1}\cdots x_{n,1}).\\
\end{align*}

\begin{proposition}\label{prop:nilKnGeNilKn1NilA}
Let $A$ be a \cdga\ and $K_n$ the kernel of the $n$-multiplication morphism $A^{\otimes n}\rightarrow A$. Then for $n\ge 3$, \[\nil\ K_n\ge \nil\ K_{n-1} + \nil\ A^+.\]
\end{proposition}

\begin{proof}
Write $r=\nil\ K_{n-1}$ and $s=\nil\ A^+$. Consider $\omega\ne 0$ a product of $r$ factors in $K_{n-1}$ and $\alpha=a_1\cdots a_s\ne 0$ with $a_i\in A^+$. Then the element $(\omega\otimes 1)(a_{1,n}-a_{1,1})\cdots(a_{s,n}-a_{s,1})=(\omega\otimes\alpha)+\xi$ with $\xi\in A^{\otimes n-1}\otimes A^{<|\alpha|}$ must be non-zero. This proves that $\nil\ K_n\ge r+s$.
\end{proof}

As pointed out by B\'arbara Guti\'errez, \cite{Gutierrez14}, the inverse inequality does not hold:
\begin{example}
Consider the \cdga\ $A=\Lambda(a,b,c,d,e)/I$ generated by elements of odd degree where $I$ is the ideal generated by $ad$, $ae$, $bcd$, $bce$. The only non-zero products of length $3$ of $A$ are $abc$, $bde$ and $cde$ and there are no non-zero products of length $4$, then we have that $\nil\ A=3$ $\nil\ K_2=5$. But $\nil\ K_3=9$, since \[\omega:= \prod_{x=a}^e(x\otimes 1\otimes 1-1\otimes x\otimes 1)\prod_{x=b}^e(x\otimes 1\otimes 1-1\otimes 1\otimes x),\] is a non-zero element of $K_3^9$ because the non-zero summand $bde\otimes abc\otimes cde$ appears only once when we develop $\omega$. 
\end{example}

\begin{corollary}
Suppose $A$ is a \cdga\ satisfying $A=A^{\mbox{\emph{even}}}$. Then \[\nil\ K_n=n(\nil\ A^+).\]
\end{corollary}
\begin{proof}
For $n=2$, choose elements $x_1,\ldots,x_r\in A^+$ with $r=\nil\ A^+$ such that the product $x_1\cdots x_r$ is non-zero. Then $\prod_{i=1}^r(x_i\otimes 1-1\otimes x_i)^2$ is a non-zero element in $K_2^{2r}$. Since $\nil\  K_2\le\nil\ (A\otimes A)^+=2(\nil\ A^+)$, we get $\nil\ K_2=2(\nil\ A^+)$. The result now follows by induction and Proposition \ref{prop:nilKnGeNilKn1NilA}.
\end{proof}
Since for formal spaces $X$, $cat(X_0)=\nil\ H^+(X,\mQ)$(apply Theorem \ref{th:tcnFormal} to $*\hookrightarrow X$), Proposition \ref{prop:nilKnGeNilKn1NilA} combined with Theorem \ref{th:tcnFormal} directly implies

\begin{theorem}\label{thm3}
Let $X$ be a formal space, then for $n\ge 2$, \[\tc_n(X_0)\ge \tc(X_0)+(n-2)\cat(X_0).\]
\end{theorem}

Recall that the wedge of two formal spaces remains a formal space.
\begin{proposition}
Let $X$ and $Y$ be formal spaces, then \[\tc(X\vee Y)\ge \cat(X)+\cat(Y).\]
\end{proposition}
\begin{proof}
Let $w_1$ and $w_2$ be monomials in $H ^*(X)$ and $H ^*(Y)$ of maximal length $n$ and $m$, $w_1=a_1\cdots a_n$ and $w_2=b_1\cdots b_m$. Then, using the notation $a^-=1\otimes a-a\otimes 1$, the identity \[a_1^-\cdots a_n^-b_1^-\cdots b_m^-=w_1\otimes w_2\pm w_2\otimes w_1\] shows that $\tc(X\vee Y)\ge n+m$.
\end{proof}

\begin{example}
Let $X$ be the wedge $(S^3\times S^3)\vee (S^3\times S^3)$. Then $\tc(S^3\times S^3)=2$ but $\tc(X)=4$.
\end{example}
The cohomology of $X$ is $\Lambda(x,y,z,t)/(xz,xt,yz,yt)$, with $x,y,z,t$ in degree 3. Then writing $x^-=1\otimes x-x\otimes 1$ and so on, we see that $x^-y^-z^-t^-=xy\otimes zt+zt\otimes xy$. This shows that $\nil\ \ker \mu_2=4$ and $\tc(X)\ge 4$. On the other hand, $\tc(X)\le \cat(X\times X)=4$.\\

We now generalize Theorem 1.4 in \cite{Jessup12}.

\begin{theorem}\label{th:TcRatPiOddDegrees}
If $\pi_*(X)\otimes\mQ$ is finite dimensional and concentrated in odd degrees. Then $\tc_n(X_0)=(n-1)\cat(X_0)$.
\end{theorem}

\begin{proof}
Let $(A,d)=\sul{\pa x_1,\ldots,x_r\pb}$ be a model of $X$ with $x_i$ in odd degree. Then \[(A,d)^{\otimes n}=\sul{\pa x_{1,1},\ldots, x_{r,1},x_{1,2},\ldots, x_{2,r},\ldots,x_{1,n},\ldots, x_{r,n} \pb}\] and $K$ the kernel of the $n$ multiplication is generated by the elements \[\set{x_{i,j}-x_{i,1}:\ 1\le i\le r,\ 2\le j\le n}.\] Since the square of these elements is zero, we have that $K^{r(n-1)+1}=0$ and so $\tc_n(X)\ge r(n-1)$. 

Now consider the pullback diagram
\[\xymatrix{
PX\ar@{}[dr]|{pb}\ar[r]\ar[d]_{q}&X^{[0,1]}\ar[d]^{p}\\
X^{n-1}\times * \ar[r] & X^n\\
}\]
where $p(\omega)=\pa \omega(0),\omega(\frac{1}{n-1}),\omega(\frac{2}{n-1}),\ldots,\omega(1)\pb$. Since $\tc_n(X)=\secat(p)$ (Proposition \ref{prop1}), $\cat(X^{n-1})=\secat(q)$ and $\secat(q)\le \secat(p)$, we have \[\tc_n(X_0)\le r(n-1)=\nil(A^{\otimes n-1}) =\cat(X_0^{n-1})\le \tc_n(X_0).\]
\end{proof}

A similar result can be found in \cite{Lupton13} for integral H-spaces.\\

\bibliography{bibliography.bib}{}
\bibliographystyle{plain}

\end{document}